\documentclass[12pt,reqno,a4paper]{amsart}
\usepackage{amsmath,amssymb,amsthm,amsaddr}

\usepackage[margin=2.5cm,top=4cm,bottom=3cm,footskip=1cm,headsep=1cm]{geometry}
\usepackage[utf8]{inputenc}
\usepackage{graphicx}
\usepackage{xcolor}
\usepackage[english]{babel}

\def\RR{\mathbb{R}}
\def\NN{\mathbb{N}}
\def\div{\mathop{\rm div}}

\def\eps{\varepsilon}

\author{H. Egger$^*$ \and T. Seitz$^*$ \and C. Tropea$^\dag$}
\address{$^*$Institute for Numerical Analysis and Scientific Computing,\\Department of Mathematics, TU Darmstadt, Germany}
\address{$^\dag$Institute for Fluid Mechanics and Aerodynamics,\\ Center of Smart Interfaces, TU Darmstadt, Germany}
\email{egger@mathematik.tu-darmstadt.de}
\email{seitz@mathematik.tu-darmstadt.de}
\email{c.tropea@sla.tu-darmstadt.de}

\title[Enhancement of flow measurements using fluid-dynamic constraints]{Enhancement of flow measurements \\[1ex]using fluid-dynamic constraints}

\newtheorem{lemma}{Lemma}[section]

\newtheorem{theorem}[lemma]{Theorem}

\theoremstyle{definition}
\newtheorem{remark}[lemma]{Remark}
\newtheorem*{example*}{Example}

\def\u{\mathbf{u}}
\def\v{\mathbf{v}}
\def\w{\mathbf{w}}
\def\z{\mathbf{z}}
\def\f{\mathbf{f}}
\def\g{\mathbf{g}}
\def\h{\mathbf{h}}
\def\n{\mathbf{n}}

\def\p{\mathrm{p}}
\def\q{\mathrm{q}}
\def\l{\mathrm{L}}

\def\I{\mathbf{I}}
\def\H{\mathbf{H}}
\def\L{\mathbf{L}}

\def\ud{{\u^{\delta}}}
\def\ua{{\u_{\alpha}}}
\def\pa{{\p_{\alpha}}}
\def\fa{{\f_{\alpha}}}
\def\ga{{\g_{\alpha}}}
\def\ha{{\h_{\alpha}}}
\def\udag{{\u^{\dag}}}
\def\fdag{{\f^{\dag}}}
\def\gdag{{\g^{\dag}}}
\def\hdag{{\h^{\dag}}}
\def\pdag{\p^\dag}

\def\zero{\mathbf{0}}

\def\div{\nabla \cdot}
\def\Div{\mathrm{div}}

\def\dO{{\partial\Omega}}
\def\dOin{{\dO_{in}}}
\def\dOout{{\dO_{out}}}
\def\dOwall{{\dO_{wall}}}

\def\ttu{\mathsf{u}}
\def\ttp{\mathsf{p}}
\def\ttf{\mathsf{f}}
\def\ttg{\mathsf{g}}
\def\tth{\mathsf{h}}
\def\ttr{\mathsf{r}}
\def\ttA{\mathsf{A}}
\def\ttB{\mathsf{B}}
\def\ttC{\mathsf{C}}

\def\ttE{\mathsf{E}}
\def\ttG{\mathsf{G}}
\def\ttH{\mathsf{H}}
\def\ttK{\mathsf{K}}
\def\ttM{\mathsf{M}}
\def\ttN{\mathsf{N}}
\def\ttK{\mathsf{K}}
\def\ttR{\mathsf{R}}

\begin{document}

\begin{abstract} 
Novel experimental modalities acquire spatially resolved velocity measurements for steady state and transient flows which are of interest for engineering and biological applications. One of the drawbacks of such high resolution velocity data is their susceptibility to measurement errors. 
In this paper, we propose a novel filtering strategy that allows enhancement of noisy measurements to obtain reconstruction of smooth divergence free velocity and corresponding pressure fields, which together approximately comply to a prescribed flow model. The main step in our approach consists of the appropriate use of the velocity measurements in the design of a linearized flow model which can be shown to be well-posed and consistent with the true velocity and pressure fields up to measurement and modeling errors. The reconstruction procedure is formulated as a linear quadratic optimal control problem and the resulting filter has analyzable smoothing and approximation properties. We also discuss briefly the discretization of our approach by finite element methods and comment on the efficient solution of the linear optimality system by iterative solvers. The capability of the proposed method to significantly reduce data noise is demonstrated by numerical tests in which we also compare to other methods like smoothing and solenoidal filtering. 
\end{abstract}

\maketitle

\begin{quote}
\end{quote}

\begin{quote}
\noindent
{\small {\bf AMS-classification:} 49J20, 35R30, 65J20, 76D55}
\end{quote}

\section{Introduction} \label{sec:intro} \setcounter{equation}{0}

Since numerous years the visualization of flow fields has had a significant impact on the systematic understanding and development of fluid-dynamic models as well as on the calibration and verification of computational methods. 
While traditional experimental techniques were able to provide only partial information about the flow field, novel measurement techniques such as particle tracking, tomographic particle imaging, or magnetic resonance velocimetry deliver spatially resolved three-dimensional velocity measurements  \cite{ElkinsMarklPelcEaton03,ElkinsAlley07,herrmann04,pereira00}. 
These new methods therefore allow to image complex flow patterns in a wide range of engineering applications and even in biological in-vivo studies.

Distributed flow measurements provide valuable information about simple and complex flows, but they are typically contaminated by measurement errors which limit their usability in practice to some extent. In order to make the flow measurements more suitable for further analysis, e.g., for model discrimination or for the assessment of derived quantities like pressure drop or wall shear stress, some sort of data post-processing is required \cite{liburdy92}. 

A widely used technique is \emph{data smoothing} which can be accomplished, for instance, by Tikhonov regularization \cite{engl96,tikhonov63} given by
\begin{align} \label{eq:A}
 \min_\u \| \u - \ud\|^2 + \alpha \|\nabla \u\|^2.
\end{align}
Here and below $\ud$ denotes the flow measurements and the minimizer will be the enhanced velocity field. Note that the penalization of the gradient term leads to a smoothed reconstruction and the choice of the regularization parameter $\alpha$ allows a certain trade-off between smoothness and fit to the data.
The underlying quadratic minimization problem can be solved efficiently by Fourier transform or multigrid iterative solvers which makes this filter very efficient in practice.
Note that the above procedure and also various other imaging methods \cite{Scherzer09} successfully reduce high frequency components in the noisy measurements but do not utilize any information about the underlying physics. 

In many applications, the fluid under consideration is incompressible and one might want to incorporate such prior knowledge into the reconstruction process. Requiring the improved velocity field to be divergence free and using a smoothing procedure similar to above, 
we obtain a constrained minimization problem of the form
\begin{subequations}
\begin{align} 
 \min_\u &\| \u - \ud\|^2 + \alpha \|\nabla \u\|^2 \qquad \text{s.t.} \label{eq:B1}\\
         &\div \u = 0.                      \label{eq:B2}
\end{align}
\end{subequations}
This quadratic minimization problem can again be solved efficiently by iterative methods. 
Various computational strategies leading to related divergence free reconstructions have been investigated recently  in the literature under the name \emph{divergence-free} or \emph{solenoidal filtering}; see e.g. \cite{busch13,daSilva13,macedo08,macnally11,OngUecker15,SchiavazziColetti14}. 
%
%
Let us note that, although some noise reduction has been observed even for the case $\alpha=0$, 
the divergence constraint alone does not formally guarantee smoothness of the reconstruction.
This can be seen from the Helmholtz decomposition of vector fields \cite{girault79} and will be illustrated by numerical tests below.

A natural extension of the solenoidal filtering approach, which takes into account only the mass conservation, would be to incorporate also a model for the momentum balance into the reconstruction process.
Since distributed measurement techniques typically acquire time averaged data, it seems reasonable to assume steady flow conditions and to consider, as a first step, the stationary Navier-Stokes equations as the governing physical model. 
The reconstruction could then be defined via
\begin{subequations}
\begin{align}
 &\min_{\f,\u,p} \| \u - \ud\|^2 + \alpha \|\f\|^2 \qquad \text{s.t.} \label{eq:C1}\\ 
                   &-\nu \Delta \u + \u \cdot \nabla \u + \nabla \p = \f, \quad  \div \u = 0. \label{eq:C2}
\end{align}
\end{subequations}
Here and below $\nu>0$ denotes the constant viscosity parameter.
In addition to the differential equations \eqref{eq:C2}, appropriate boundary conditions have to be specified. 
%
The residual $\f$ in the momentum equation serves as a measure for the deviation from the idealized flow model 
due to unmodeled effects like time dependence or non-Newtonian behaviour.
Since a prescribed flow model is satisfied by the reconstructed fields, one may call such an approach a \emph{fluid-dynamically consistent} filter.
Due to the presence of the flow model, the reconstruction will be smooth automatically and no 
additional penalization of the velocity gradients is required.
Moreover, some information about the pressure is obtained. 

The system \eqref{eq:C1}--\eqref{eq:C2} has the form of an optimal control problem governed by the Navier-Stokes equations. 
Such problems have been investigated extensively in the literature; see e.g. \cite{desai94,gunzburger91,herzog10,hou93} for steady and \cite{abergel90,fattorini92,fursikov82,hinze01,gunzburger92} for unsteady flow. 
Note that the nonlinearity in the momentum equation poses severe challenges, both, for the analysis and for the numerical solution. 
It is well-known, for instance, that the Navier-Stokes system admits a unique solution only for sufficiently small data \cite{foias77,temam84}. 
Moreover, due to the nonlinear constraints, the optimization problem \eqref{eq:C1}--\eqref{eq:C2} is non-convex and may have many local minima.
Both aspects make the computational solution demanding or even infeasible. 

In this paper, we therefore propose a strategy that allows us to take advantage of the benefits and at the same time to overcome the drawbacks in the previous approach.
The basic step is to use the distributed velocity measurements in order to replace the nonlinear term in the momentum equation by some linearization; one may think of $\ud \cdot \nabla \u$ as an approximation for $\u \cdot \nabla \u$, although such a simple choice would not yield a well-posed problem in general due to lack of smoothness in the data. 
However, a proper linearization of the convective term will $\u \cdot \nabla \u$ allow us to replace the nonlinear problem \eqref{eq:C1}--\eqref{eq:C2} by a linear quadratic optimal control problem with a unique minimizer that can be computed efficiently.
The use of the distributed measurements in the governing equations is closely related to the \emph{equation error method}, which is well-established in the context of parameter estimation \cite{banks89,hanke99}.

In summary we thus obtain a well-posed and analyzable reconstruction method that produces 
a smooth divergence free velocity field and a corresponding pressure distribution which together 
approximately satisfy the prescribed fluid-dynamic model and at the same time agree well with the measurements.
A proper choice of the regularization parameter $\alpha$ will allow us to find a good balance between data fit and model errors.

\bigskip 

The remainder of the manuscript is organized as follows:
In Section~\ref{sec:model}, we introduce the linearized fluid flow model
underlying our reconstruction approach and we formulate appropriate boundary conditions. 
We then establish the well-posedness of the linearized flow problem and derive error estimates for the linearization procedure.
In Section~\ref{sec:method}, we introduce the linearized optimal control problem 
which is the mathematical formulation of our reconstruction procedure. 
We prove existence and uniqueness of minimizers,
provide some error estimates, and highlight a direct connection to the solenoidal filtering.
Our approach is formulated in infinite dimensions and some discretization strategy is required in order to obtain implementable algorithms.
In Section~\ref{sec:realization}, we therefore outline the systematic discretization by finite element methods 
and we briefly discuss efficient strategies for the numerical solution of the discretized optimal control problem.
The viability of our approach is illustrated in Section~\ref{sec:computation} by some computational tests in which we also compare with the smoothing and the solenoidal filtering approaches outlined above. 
The presentation concludes with a short summary and a discussion of open problems and possible directions for further research.

\section{The linearized flow model} \label{sec:model}

Let us first introduce the linearized flow model that is used as a constraint in the reconstruction process 
and establish its well-posedness. 
We use the fact that the true flow field satisfies a model of similar structure
and derive some perturbation error estimates. 
For illustration, we discuss in some detail the Poiseuille flow between two parallel plates.

\subsection{Geometric setting} \label{sec:geom}
We start with fixing the geometric setting we have in mind. 
Let $\Omega \subset \RR^d$, $d=2,3$ be some bounded Lipschitz domain. We assume that the boundary $\partial\Omega$ 
is piecewise smooth and can be split into three distinct parts $\dOwall$, $\dOin$, and $\dOout$ such that $\overline \dOin \cap \overline\dOout = \emptyset$. 
One may think of a channel where $\dOin$ is the inflow, $\dOout$ the outflow, and $\dOwall$ the wall of the channel. 

\subsection{The linearized flow model} \label{sec:lin}
The key step in our approach is to replace the nonlinear convective term in \eqref{eq:C2} by
an appropriate linearization. For this, we make use of the following identity. 
Let $\u$ and $\w$ be two smooth vector fields and $\div \w=0$. Then 
\begin{align} \label{eq:id}
\w \cdot \nabla \u = \Div (\u \otimes \w) = \tfrac{1}{2}  \w \cdot \nabla \u + \tfrac{1}{2} \Div ( \u \otimes \w), 
\end{align}
where $(\w \cdot \nabla \u)_i = \sum\nolimits_j \w_j \partial_j \u_i$ and $\Div (\u \otimes \w)_i = \sum\nolimits_j \partial_j (\u_i \w_j)$ by definition. 
Throughout we use bold symbols to denote vector valued functions and spaces of such functions. 

For an incompressible fluid, we can then express the convective term $\u \cdot \nabla \u$ equivalently by $\frac{1}{2} \u \cdot \nabla \u + \frac{1}{2} \Div(\u \otimes \u)$. Such a form of the convective term is sometimes employed in the design and analysis of numerical methods for incompressible flow. 
Using the velocity measurements to replace one of the functions in either of the quadratic terms,
we obtain $\frac{1}{2} \ud \cdot \nabla \u + \frac{1}{2} \Div(\u \otimes \ud)$ as an approximation. The latter expression can then be used to replace the nonlinear convective term $\u \cdot \nabla \u$ in the momentum equation \eqref{eq:C2}, which leads to the following linearized flow model
\begin{subequations}
\begin{align}
-\nu \Delta \u + \tfrac{1}{2} \u^\delta \cdot \nabla \u 
    + \tfrac{1}{2} \Div( \u \otimes \u^\delta) + \nabla \p &= \f \qquad \text{in } \Omega, \label{eq:lin1}\\
                                                  \div \u &= 0  \qquad \text{in } \Omega. \label{eq:lin2}
\end{align}
Apart from the special form of the convective term this amounts to an Oseen problem 
with convective velocity $\ud$ that will neither be smooth nor divergence free in general.
To complete the description of the model, we impose the following boundary conditions
\begin{align}
                                                                  \u &= \g    \qquad \text{on } \dOin, \label{eq:lin3}\\
                                                                  \u &= \zero \qquad \text{on } \dOwall, \label{eq:lin4}\\
(-\nu \nabla \u + \tfrac{1}{2} \u \otimes \u^\delta + \p \I) \cdot \n &= \h    \qquad \text{on } \dOout. \label{eq:lin5}
\end{align}
\end{subequations}
We thus prescribe the full velocity field at the inflow boundary
and use a Neumann-type boundary condition at the outflow. 
A no-slip condition is used at the walls of the channel. 
Also other types of the boundary conditions could be incorporated with minor changes.
%
%
The functions $\f$, $\g$, and $\h$ arising as data in the flow model will later enter the reconstruction process 
as additional parameters which are to be determined.

\subsection{Well-posedness of the linearized flow model} \label{sec:selposedlin}

Since the data $\ud$ stem from measurements, one can in general not require their spatial smoothness.
It is therefore not clear a-priori, if the model \eqref{eq:lin1}--\eqref{eq:lin5} is meaningful
from a mathematical point of view. 
As a first step, we thus want to clarify the well-posedness of the linearized flow model. 
\begin{theorem} \label{thm:1}
Let $\ud \in \L^3(\Omega)$.
Then for 
$\f \in \L^2(\Omega)$, $\g \in \H_0^1(\dOin)$, and $\h \in \L^2(\dOout)$,
the problem \eqref{eq:lin1}--\eqref{eq:lin5} has a unique weak solution $\u \in \H^1(\Omega)$ and $\p \in \l^2(\Omega)$.
Moreover 
\begin{align*} 
\|\u\|_{\H^1(\Omega)} + \|\p\|_{\l^2(\Omega)} 
\le C \big( \|\f\|_{\L^2(\Omega)} + \|\g\|_{\H^1(\dOin)} + \|\h\|_{\L^2(\dOout)}\big)
\end{align*}
with $C$ depending only on $\|\ud\|_{\L^3(\Omega)}$, on the parameter $\nu$, and on the geometry.
\end{theorem}
\begin{proof}
The result follows with standard arguments for the analysis of the stationary flow equations. Since the momentum equation is a bit non-standard, we sketch the main steps of the proof.
The weak solution of problem \eqref{eq:lin1}--\eqref{eq:lin5} is characterized by the following 
mixed variational problem: Find $\u \in \H^1(\Omega)$ and 
$\p \in \l^2(\Omega)$ with $\u = \g$ on $\dOin$ and $\u=\zero$ on $\dOwall$ such that 
\begin{align} 
a(\u,\v) + c(\ud;\u,\v) + b(\v,\p) &= (\f,\v)_\Omega + (\h,\v)_{\dOout} \label{eq:var1}\\
b(\u,\q) &= 0 \label{eq:var2}
\end{align}
for all $\v \in \H^1(\Omega)$ and  $\q \in L^2(\Omega)$ with $\v=\zero \text{ on } \dOwall \cup \dOin$. 
Here $a(\u,\v) = \nu (\nabla \u, \nabla \v)_\Omega$ is the bilinear form for the viscous term, 
$c(\ud;\u,\v) = \tfrac{1}{2} (\ud \cdot \nabla \u,\v)_\Omega - \tfrac{1}{2} (\u, \ud \cdot \nabla \v)_\Omega$ 
represents the convective term, and $b(\u,\q) = -(\div \u,\q)_\Omega$ the weak form for the divergence operator. 
Without loss of generality, we may assume that $\g=\zero$ in the sequel.
%
Due to the special form of the convective terms, the form $c(\ud;\cdot,\cdot)$ is anti-symmetric, which implies $c(\ud;\u,\u)=0$.
The assumption $\ud \in \L^3(\Omega)$ further implies that $c(\ud;\u,\v)$ is bounded for $\u,\v \in \H^1(\Omega)$.
Standard arguments used for the analysis of the Oseen problem then yield the assertions; see \cite[Ch.~II]{temam84} for details.
\end{proof}
As can be seen from the proof, the special form of the convective term and of the outflow boundary condition were essential here to obtain the well-posedness and the energy estimate under minimal regularity assumptions on the measured flow field.

\subsection{Estimates for the linearization error} \label{sec:linerr}

The linearization procedure introduces some perturbations which we would like to quantify next.
To be able to do so, we set up a flow model of similar structure which describes the true flow. 
Let $\udag$ and $\pdag$ denote the true velocity and pressure fields which are assumed to be sufficiently smooth. Then 
\begin{subequations}
\begin{align} 
-\nu \Delta \udag + \tfrac{1}{2} \udag \cdot \nabla \udag  + \tfrac{1}{2} \Div(\udag \otimes \udag) + \nabla \pdag &= \fdag, \label{eq:ex1}\\
\div \udag &= 0, \label{eq:ex2}
\end{align}
for some appropriate function $\fdag$ which is just defined as the left hand side of the first equation. 
In a similar way, we can define functions $\gdag$ and $\hdag$ such that
\begin{align}
                                                                 \udag &= \gdag    \qquad \text{on } \dOin, \label{eq:ex3}\\
                                                                 \udag &= \zero \  \qquad \text{on } \dOwall, \label{eq:ex4}\\
(-\nu \nabla \udag + \tfrac{1}{2} \udag \otimes \udag + \pdag \I) \cdot \n &= \hdag    \qquad \text{on } \dOout. \label{eq:ex5}
\end{align}
\end{subequations}
This system has the same form as \eqref{eq:lin1}--\eqref{eq:lin5} but with data $\f$, $\g$, $\h$
and convective velocity field $\ud$ replaced appropriately. 
This allows to estimate the difference between the solutions of \eqref{eq:ex1}--\eqref{eq:ex5} and the linearized model \eqref{eq:lin1}--\eqref{eq:lin5}.
\begin{theorem} \label{thm:2}
Let $\ud \in \L^3(\Omega)$ and let $(\u,\p)$ and $(\udag,\pdag)$ be defined as above. 
Then  
\begin{align*}
&\|\u - \udag\|_{\H^1(\Omega)} + \|\p - \pdag\|_{\l^2(\Omega)} \\
& \qquad \qquad \le C \big( \|\f-\fdag\|_{\L^2(\Omega)} + \|\g-\gdag\|_{\H^1(\dOin)} + \|\h-\hdag\|_{\L^2(\dOout)} + \|\udag-\ud\|_{\L^3(\Omega)}\big) 
\end{align*}
with $C$ depending only on the bounds for the data, the parameter $\nu$, and the geometry.
\end{theorem}
\begin{proof}
Let $(\widetilde \u,\widetilde \p)$ denote the solution of problem \eqref{eq:lin1}--\eqref{eq:lin5} with $\ud$ replaced by $\udag$. 
The error can then be decomposed into 
$$
\u - \udag = \w + \z \qquad \text{and} \qquad \p - \pdag = \pi + \psi,
$$
with functions $\w=\u-\widetilde\u$, $\pi=\p-\widetilde \p$, $\z=\widetilde \u - \udag$, and $\psi=\widetilde \p - \pdag$ that can be estimated separately. 
From the definition of $\w$ and $\pi$, we observe that $\w=\zero$ on $\dOin \cup \dOwall$ and
\begin{align*}
a(\w,\v) + c(\ud;\w,\v) + b(\v,\pi) &= c(\udag-\ud;\widetilde\u,\v) \\
b(\w,\q) &= 0
\end{align*}
for all $\v \in \H^1(\Omega)$ and $\q \in L^2(\Omega)$ with $\v=0 \text{ on } \dOwall \cup \dOin$.
Choosing  $\v=\w$ and $\q=-\pi$ as test functions and applying the Poincar\'e-Friedrichs inequality yields
\begin{align*}
c \|\w\|^2_{\H^1(\Omega)} 
\le a(\w,\w) 
&= a(\w,\w) + c(\ud;\w,\w) + b(\w,\pi) \\
&= c(\udag-\ud;\widetilde\u,\w) \le C \|\udag-\ud\|_{\L^3(\Omega)} \|\widetilde \u\|_{\H^1(\Omega)} \|\w\|_{\H^1(\Omega)}.
\end{align*}
Since the term $\|\widetilde \u\|_{\H^1(\Omega)}$ can be bounded uniformly by Theorem~\ref{thm:1}, we further obtain 
$\|\w\|_{\H^1(\Omega)} \le C' \|\udag-\ud\|_{\L^3(\Omega)}$
and $\|\pi\|_{L^2(\Omega)}$ can be bounded by $\|\w\|_{\H^1(\Omega)}$ with the usual arguments. 
Next observe that $\z=\g-\gdag$ on $\dOin$, $\z=\zero$ on $\dOwall$, and 
\begin{align*} 
a(\z,\v) + c(\udag;\z,\v) + b(\v,\psi) &= (\f-\fdag,\v)_\Omega + (\h-\hdag,\v)_{\dOout} \\
b(\z,\q) &= 0
\end{align*}
for all $\v \in \H^1(\Omega)$ and $\q \in L^2(\Omega)$ with $\v=0 \text{ on } \dOwall \cup \dOin$. 
By Theorem~\ref{thm:1} we thus obtain $\|\z\|_{\H^1(\Omega)} + \|\psi\|_{L^2(\Omega)} 
\le C'' \big( \|\f-\fdag\|_{\L^2(\Omega)} + \|\g-\gdag\|_{\H^1(\dOin)} + \|\h-\hdag\|_{\L^2(\dOout)}\big)$. 
The assertion of Theorem~\ref{thm:2} then follows by a combination of these  estimates.
\end{proof}

\begin{remark}
If the flow model is a reasonable approximation for the physical conditions, we may assume that $\f^\dag$, $\g^\dag$, and $\h^\dag$ are known to first order. For illustration, let us discuss a particular example which will also serve as our test problem later on.
\end{remark}

\subsection{Poisseuille flow} \label{sec:poisseuille}

In simple geometries the solution of the stationary Navier-Stokes equations can be computed analytically. The laminar flow between two parallel plates with distance $d$ along a path of length $L$, for instance, is characterized by
\begin{align}  \label{eq:poisseuille}
\p^\dag(x,y)= \p_0 + \frac{\p_L-\p_0}{L} x
\qquad \text{and} \qquad 
\u^\dag(x,y) = \left(\frac{\p_L-\p_0}{2\nu L} (dy - y^2), 0\right),
\end{align}
where $\p_0$, $\p_L$ denote the pressures at position $x=0$ and $x=L$, respectively. 
Similar formulas are available for channels with other geometries \cite{batchelor67}.
The solution $(\udag,\pdag)$ given by the Poisseuille law \eqref{eq:poisseuille} satisfies the system \eqref{eq:ex1}--\eqref{eq:ex5} with $\fdag=\zero$ and functions $\gdag$ and $\hdag$ that can
be computed from \eqref{eq:poisseuille}.
%
%
The above estimate for the perturbation introduced by the linearization procedure reads
\begin{align*}
&\|\u - \udag\|_{\H^1(\Omega)} + \|\p - \pdag\|_{\L^2(\Omega)} \\
& \qquad \le C \big( \|\f-\fdag\|_{\L^2(\Omega)} + \|\g-\gdag\|_{\H^1(\dOin)} + \|\h-\hdag\|_{\L^2(\dOout)} + \|\udag-\ud\|_{\L^3(\Omega)}\big).
\end{align*}
The total error in the solution thus results from misspecifications $\f-\fdag$, $\g-\gdag$, $\h-\hdag$ 
of the physical model on the one hand, and from perturbations $\ud - \udag$ in the measurements on the other. 
This observation will be the guideline for the formulation of our reconstruction method in the next section.

\section{The reconstruction method} \label{sec:method} \setcounter{equation}{0}

For the enhancement of the velocity measurements $\ud$, 
we now consider the following linearization of the optimal control approach \eqref{eq:C1}--\eqref{eq:C2} 
outlined in the introduction.
\begin{subequations}
\begin{align}
 &\min_{\f,\g,\h,\u,\p} \|\u - \ud\|^2_{\L^2(\Omega)} 
+ \alpha \big( \|\f-\f^*\|_{\L^2(\Omega)}^2 + \|\g-\g^*\|_{\H^1(\dOin)}^2 + \|\h-\h^*\|_{\L^2(\dOout)}^2\big) \label{eq:min1}\\
&\qquad \qquad \text{s.t. } \eqref{eq:lin1}-\eqref{eq:lin5}. \label{eq:min2}
\end{align} 
\end{subequations}
\noindent
The choice of function spaces over which is minimized is clear from the theorems of the previous section and the norms in the regularization terms. 
The functions $\f^*$, $\g^*$, and $\h^*$ serve as approximations for the unknown correct data 
$\fdag$, $\gdag$, and $\hdag$ in the governing fluid-dynamic model \eqref{eq:lin1}--\eqref{eq:lin5} and enter as additional \emph{model parameters}.
The reconstructed field thus minimizes a weighted sum of deviations from the velocity data and the prescribed flow model and the choice of the regularization parameters $\alpha$ allows us to balance the two error contributions.

\subsection{Existence of a unique minimizer} \label{sec:wellposedrec}

Due to the well-posedness of the linearized flow problem \eqref{eq:lin1}--\eqref{eq:lin5},
we can express $\u=\u(\f,\g,\h)$ and $\p=\p(\f,\g,\h)$ in terms of the functions $\f$, $\g$, and $\h$. 
This allows us to eliminate the fields $\u$ and $\p$ from \eqref{eq:min1}--\eqref{eq:min2} 
and to obtain the following equivalent minimization problem
\begin{align}
\min_{\f,\g,\h} J_\alpha(\f,\g,\h)
\label{eq:red}
\end{align}
with reduced cost functional $J_\alpha$ depending only on the data $\f$, $\g$ and $\h$, which is defined by 
$J_\alpha(\f,\g,\h) = \|\u(\f,\g,\h)-\ud\|^2_{\L^2(\Omega)} + \alpha \big( \|\f-\f^*\|_{\L^2(\Omega)}^2 + \|\g-\g^*\|_{\H^1(\dOin)}^2 + \|\h-\h^*\|_{\L^2(\dOout)}^2\big).$ 
The existence of a unique minimizer for \eqref{eq:red} follows with standard arguments in convex analysis. 
By the equivalence with the problem \eqref{eq:min1}--\eqref{eq:min2}, we also obtain the well-posedness of the original formulation.
\begin{theorem} \label{thm:3}
Let $\ud \in \L^3(\Omega)$ and let $\f^* \in \L^2(\Omega)$, $\g^* \in \H_0^1(\dOin)$ and $\h^* \in \L^2(\dOout)$. 
Then for any $\alpha > 0$, the reduced problem \eqref{eq:red} has a unique solution with components $\fa \in\L^2(\Omega)$, $\ga \in \H_0^1(\dOin)$, and $\ha \in \L^2(\dOout)$.
Together with $\ua=\u(\fa,\ga,\ha)$ and $\pa=\p(\fa,\ga,\ha)$ this yields the unique solution of problem \eqref{eq:min1}--\eqref{eq:min2}.
\end{theorem}
\begin{proof}
The mapping $(\f,\g,\h) \mapsto (\u(\f,\g,\h),\p(\f,\g,\h))$ is affine linear and continuous. 
As a consequence, the functional $J_\alpha$ is quadratic, bounded from below, strictly convex, 
lower semi-continuous, and coercive. This implies existence of a unique minimizer. 
\end{proof}
As we will see below, both formulations \eqref{eq:min1}--\eqref{eq:min2} as well as \eqref{eq:red} 
are well suited as a starting point for the design of efficient numerical solution procedures. 

\subsection{Estimates for the reconstruction error} \label{sec:recerr}

As a theoretical justification for the proposed method let us next present some quantitative estimates for the reconstruction error
which illustrate what kind of numerical results can be expected and which allow us to draw some conclusions about the proper choice 
of the regularization parameter. 
\begin{theorem} \label{thm:4} 
Let $(\ua,\pa)$ denote the velocity and pressure components of the unique solution of problem \eqref{eq:min1}--\eqref{eq:min2} and  assume that $\|\udag - \ud\|_{\L^3(\Omega)} \le \delta$. 
Then the following estimates hold true:
\begin{itemize}
\item[(i)] $\|\ua - \udag\|_{\L^2(\Omega)}^2 \le C \delta^2 + \alpha \big( \|\f^\dag-\f^*\|^2_{\L^2(\Omega)} + \|\g^\dag - \g^*\|^2_{\H^1(\dOin)} + \|\h^\dag-\h^*\|_{\L^2(\dOout)}^2\big)$.
\smallskip
\item[(ii)] $\|\ua - \udag\|^2_{\H^1(\Omega)} \le C \big( \delta^2 + \delta^2/\alpha + \|\fdag-\f^*\|^2_{\L^2(\Omega)} + \|\g^\dag-\g^*\|^2_{\L^2(\dOin)} + \|\h^\dag-\h^*\|_{\L^2(\dOout)}^2\big)$.
\end{itemize}
The second bound also holds for the error $\|\pa-\pdag\|_{\l^2(\Omega)}^2$ in the pressure.
\end{theorem}
\begin{proof}
Let $(\widehat\u,\widehat\p)$ denote the solution of \eqref{eq:var1}--\eqref{eq:var2} with $\f$, $\g$, $\h$ replaced by 
$\fdag$, $\gdag$, and $\hdag$, respectively. 
Then $\|\widehat\u-\udag\|_{\H^1(\Omega)} + \|\widehat \p - \pdag\|_{L^2(\Omega)} \le C \|\udag-\ud\|_{\L^3(\Omega)}$,
which follows like the estimate for the function $\w$ in the proof of Theorem~\ref{thm:2}. 
By definition of $\ua$ as a minimizer, we further have 
\begin{align*}
&\|\ua-\ud\|^2_{\L^2(\Omega)} \\
&\le \|\ua-\ud\|^2_{\L^2(\Omega)} + \alpha \big( \|\fa-\f^*\|_{\L^2(\Omega)}^2 + \|\ga-\g^*\|_{\H^1(\dOin)}^2 + \|\ha-\h^*\|_{\L^2(\dOout)}^2\big)   \\
&\le \|\widehat\u-\ud\|^2_{\L^2(\Omega)} + \alpha \big( \|\fdag-\f^*\|_{\L^2(\Omega)}^2 + \|\gdag-\g^*\|_{\H^1(\dOin)}^2 + \|\hdag-\h^*\|_{\L^2(\dOout)}^2\big).  
\end{align*}
The first assertion now follows by combining the two estimates and using the assumption on the data error together with the continuous embedding of $L^3(\Omega)$ into $L^2(\Omega)$.
For the second estimate, we use the triangle inequality 
to obtain
\begin{align*}
&\|\ua-\udag\|_{\H^1(\Omega)} + \|\pa-\pdag\|_{\L^2(\Omega)} \\
& \le (\|\widehat \u-\udag\|_{\H^1(\Omega)} +  \|\widehat \p-\pdag\|_{\L^2(\Omega)}) + (\|\ua-\widehat \u\|_{\H^1(\Omega)} + \|\pa-\widehat \p\|_{\L^2(\Omega)}).
\end{align*}
The first term can be estimated by $\|\widehat \u - \udag\|_{\H^1(\Omega)}  + \|\pa - \widehat \p\|_{L^2(\Omega)}\le C \delta$ as above. 
Proceeding as in the proof of Theorem~\ref{thm:1}, the remaining term can be bounded by 
\begin{align*}
&\|\ua-\widehat\u\|_{\H^1(\dO)} + \|\pa-\widehat \p\|_{\L^2(\Omega)} \\
&\le C \big( \|\fa-\fdag\|_{\L^2(\Omega)} + \|\ga-\gdag\|_{\H^1(\dOin)} + \|\ha-\hdag\|_{\L^2(\dOout)}\big). 
\end{align*}
Let us consider the first term on the right hand side in more detail. 
Via the triangle inequality, we get $\|\fa-\fdag\|_{\L^2(\Omega)} \le \|\fa-\f^*\|_{\L^2(\Omega)} + \|\fdag-\f^*\|_{\L^2(\Omega)}$. 
Note that the second term already appears in the final result. 
By the definition of the minimizers, the first term can be further estimated by
\begin{align*}
&\|\fa-\f^*\|^2_{\L^2(\Omega)} \\
&\le \alpha^{-1} \|\ua-\ud\|_{\L^2(\Omega)}^2 + \|\fa-\f^*\|_{\L^2(\Omega)}^2 + \|\ga-\g^*\|_{\H^1(\dOin)}^2 + \|\ha-\h^*\|_{\L^2(\dOout)}^2   \\
&\le \alpha^{-1} \|\widehat\u-\ud\|^2_{\L^2(\Omega)} + \|\fdag-\f^*\|_{\L^2(\Omega)}^2 + \|\gdag-\g^*\|_{\H^1(\dOin)}^2 + \|\hdag-\h^*\|_{\L^2(\dOout)}^2.  
\end{align*}
Together with the estimates for $\widehat\u-\udag$ and the bound on the data error, this yields
\begin{align*}
\|\fa-\f^*\|^2_{\L^2(\Omega)} \le C \big(\delta^2/\alpha +   \|\fdag-\f^*\|_{\L^2(\Omega)}^2 + \|\gdag-\g^*\|_{\H^1(\dOin)}^2 + \|\hdag-\h^*\|_{\L^2(\dOout)}^2 \big).
\end{align*}
The same arguments lead to estimates for $\|\ga-\g^*\|_{\H^1(\dOin)}^2$ and $\|\ha-\h^*\|_{\L^2(\dOout)}$
which completes the proof of the second assertion.
\end{proof}

\begin{remark}
The estimates in Theorem~\ref{thm:4} show that a proper choice of the regularization parameter $\alpha$
allows to obtain a balance between data fit and model error. In particular, a good 
fit to the measurements can always be obtained by choosing $\alpha$ sufficiently small.
If the proposed flow model is a good description of the physical conditions, i.e., if the model errors $\|\f^\dag-\f^*\|$, $\|\gdag-\g^*\|$, and $\|\hdag - \h^*\|$ are sufficiently small, one can actually choose the regularization parameter in the order of one and still obtain a good fit for the velocity field in the stronger norm and also for the pressure.
\end{remark}

\subsection{Poisseuille flow} \label{sec:poisseuillerec}

Let us return to the setting discussed in Section~\ref{sec:poisseuille}.
In this case, we may choose $\f^*=\fdag$, $\g^*=\gdag$, and $\h^*=\hdag$ 
with $\fdag=\zero$ and profiles $\gdag$, $\hdag$ computed from the Poisseuille law \eqref{eq:poisseuille}. 
The estimates of the previous theorem then simplify to
\begin{align*}
\|\ua - \udag\|_{\L^2(\Omega)} \le C \delta \qquad \text{and} \qquad 
\|\ua - \udag\|_{\H^1(\Omega)} + \|\pa-\pdag\|_{L^2(\Omega)} \le C \delta (1 + 1 / \sqrt{\alpha}).
\end{align*}
The reconstruction errors will therefore be in the order of the measurement errors if we choose 
the regularization parameter $\alpha$ in the order of one!  
In particular, any fixed choice of the regularization parameter will lead to $O(\delta)$ convergence for the velocity errors in the $L^2$- and the $H^1$-norm. 
A more sophisticated choice of the regularization parameter is however required if the underlying flow model does not describe the physical situation sufficiently well, i.e., if the model data $\f^*$, $\g^*$, and $\h^*$ are not chosen appropriately.

\subsection{Further properties} \label{sec:further}

As a final step of our theoretical considerations, let us comment on two properties of the reduced cost functional $J_\alpha$ and the minimizers $\ua$. 
\begin{theorem} \label{thm:5}
Let $\ua=\u(\fa,\ga,\ha)$ be defined as above. Then
\begin{itemize}
 \item[(i)] $\min_{\f,\g,\h}J_\alpha(\f,\g,\h) \le \min_{\f,\g,\h} J_\beta(\f,\g,\h)$ whenever $\alpha \le \beta$.
\smallskip
 \item[(ii)] $\ua \stackrel{L^2}{\to} \u^{SF}$ with $\alpha \to 0$, where $\u^{SF}$ is the solution of
 \eqref{eq:B1}--\eqref{eq:B2} with $\alpha=0$. 
\end{itemize}
\end{theorem}
\begin{proof}
By definition of $\ua$ as minimizer, we have 
\begin{align*}
\min_{\f,\g,\h} J_\alpha(\f,\g,\h) 
&= J_\alpha(\fa,\ga,\ha) \le J_\alpha(\f_\beta,\g_\beta,\h_\beta) 
\le J_\beta (\f_\beta,\g_\beta,\h_\beta), 
\end{align*}
where in the last step we used the condition $\alpha \le \beta$ and the positivity of the regularization term.
This already yields the first assertion.

To show the second, let us note that for any $\eps>0$ one can find a smooth divergence free approximation $\widetilde\u^{SF}$ for the velocity field $\u^{SF}$ that is obtained by \eqref{eq:B1}--\eqref{eq:B2} with $\alpha=0$
and such that $\|\u^{SF}-\widetilde\u^{SF}\|_{\L^2(\Omega)} \le \eps$.
By plugging $\widetilde \u^{SF}$ into \eqref{eq:lin1}--\eqref{eq:lin5}, we obtain corresponding 
residuals $\widetilde \f$, $\widetilde \g$, and $\widetilde \h$. Using the definition of $\ua$, we further get 
\begin{align*}
&\|\ua-\ud\|^2_{\L^2(\Omega)} \\
&\le \|\ua-\ud\|^2_{\L^2(\Omega)} + \alpha \big( \|\fa-\f^*\|_{\L^2(\Omega)}^2 + \|\ga-\g^*\|_{\H^1(\dOin)}^2 + \|\ha-\g^*\|_{\L^2(\dOout)}^2\big)  \\
&\le \|\widetilde\u^{SF}-\ud\|^2_{\L^2(\Omega)} + \alpha \big( \|\widetilde\f-\f^*\|_{\L^2(\Omega)}^2 + \|\widetilde\g-\g^*\|_{\H^1(\dOin)}^2 + \|\widetilde\h-\h^*\|_{\L^2(\dOout)}^2\big).
\end{align*}
This shows that 
\begin{align*}
\lim\sup_{\alpha \to 0} \|\ua - \ud\|_{\L^2(\Omega)} \le \|\widetilde\u^{SF} - \ud\|_{\L^2(\Omega)} \le \|\u^{SF} - \ud\|_{\L^2(\Omega)}+\eps
\end{align*}
for any $\eps>0$. 
We therefore conclude that $\lim\sup_{\alpha \to 0} \|\ua - \ud\|_{\L^2(\Omega)} \le \|\u^{SF} - \ud\|_{\L^2(\Omega)}$. 
Since $\|\ua\|_{L^2(\Omega)} \le C$ for all $\alpha>0$, we can select a subsequence $\{\u_{\alpha'}\}$ converging weakly to some $\bar \u \in \L^2(\Omega)$, and by the lower semi-continuity of the norm, we can deduce
\begin{align*}
\|\bar \u - \ud\|_{\L^2(\Omega)} 
\le \lim\inf_{\alpha'\to 0}  \|\u_{\alpha'} - \ud\|_{\L^2(\Omega)} 
\le \lim\sup_{\alpha' \to 0} \|\u_{\alpha'} - \ud\|_{\L^2(\Omega)}\le \|\u^{SF} - \ud\|_{\L^2(\Omega)}
\end{align*}
Moreover, since $\div \ua=0$ for all $\alpha>0$, we also have $\div \bar \u=0$. 
But since $\u^{SF}$ is the unique minimizer of \eqref{eq:B1}--\eqref{eq:B2}, we may conclude that $\bar \u = \u^{SF}$.
\end{proof}
\begin{remark}
For $\alpha \to 0$, not only the minimal values of the cost functional $J_\alpha$ but also the data residuals $\|\ua-\ud\|_{\L^2(\Omega)}$ can be shown to decrease monotonically. This allows us to choose the regularization parameter $\alpha$ via a \emph{discrepancy principle}. 
The fact that the reconstructed velocity fields $\ua$ converges to the solution $\u^{SF}$ of the solenoidal filtering problem with $\alpha \to 0$ indicates that our reconstruction approach yields an 
enhancement of and is at least as stable as the solenoidal filtering \eqref{eq:B1}--\eqref{eq:B2} with $\alpha=0$.
\end{remark}

\section{Numerical realization} \label{sec:realization} \setcounter{equation}{0}

In order to obtain computational algorithms, we still have to discretize
the reconstruction method proposed in this paper.
Since the numerical approximation of optimal control problems is well-understood, 
we only sketch the main ideas here.   

\subsection{Discretization of the fluid-dynamic model} \label{sec:fem}

For discretization of the state system \eqref{eq:lin1}--\eqref{eq:lin5}, we use a standard Galerkin method with an inf-sup stable pair of finite element spaces \cite{boffi08,girault79}. 
This leads to an algebraic system of the form 
\begin{subequations}
\begin{align}
\ttA \ttu  + \ttB^\top \ttp &= \ttM \ttf + \tfrac{1}{\eps} \ttR \ttE \ttg + \ttN \tth, \label{eq:lin1h}\\
                  \ttB \ttu &= 0,                                                \label{eq:lin2h}         
\end{align}
\end{subequations}
where $\ttA = \nu \ttK + \ttC(\ttu^\delta) + \tfrac{1}{\eps} \ttR$,
the matrix $\ttK$ represents the vector Laplacian, $\ttM$ the mass matrix,
$\ttC(\ttu^\delta)$ is the anti-symmetric convective term, 
and $\ttB$ is the discrete divergence operator. The Dirichlet boundary conditions are incorporated here by a penalty approach with $\eps$ being a small parameter. The matrix $\ttR$ represents the corresponding 
integrals over the boundary $\dOin \cup \dOwall$, and $\ttE$ realizes an extension of boundary values into the domain by zero. Similarly, the matrix $\ttN$ accounts for integrals over $\dOout$. The vector $\ttu^\delta$ denotes the finite element representation of the velocity measurements.
When using a finite element discretization, all matrices will be sparse.
For details on the implementation, we refer to standard textbooks \cite{girault79,temam84}.

\subsection{Discretization of the optimal control problem} \label{sec:disc}

Using the above notation, the discretized optimal control problem can be written as
\begin{subequations}
\begin{align} 
&\min_{\ttu,\ttp,\ttf,\ttg,\tth} \|\ttu - \ttu^\delta\|_\ttM^2 + \alpha \big( \|\ttf-\ttf^*\|_\ttM^2 + \|\ttg-\ttg^*\|_\ttG^2 + \|\tth-\tth^*\|_\ttH^2\big) \label{eq:min1h}\\
&\qquad \text{s.t. } \eqref{eq:lin1h}-\eqref{eq:lin2h}.   \label{eq:min2h}
\end{align}
\end{subequations}
Here $\|\ttr\|_\ttM^2 = \ttr^\top \ttM \ttr$ denotes the norm induced by the positive definite matrix $\ttM$. 
The Gramian matrices $\ttG$ and $\ttH$ induce the corresponding norms on the boundary.

The first order optimality conditions for \eqref{eq:min1h}--\eqref{eq:min2h} are obtained by differentiating the associated Lagrangian. Since the problem under investigation is quadratic and strictly convex, the first order optimality conditions are necessary and sufficient. 
Due to the use of finite elements, the discrete optimality system is sparse symmetric and indefinite and can be solved efficiently by appropriate preconditioned iterative methods \cite{draganescu13,schoeberl07,takacs15}.  

The discrete state system \eqref{eq:lin1h}--\eqref{eq:lin2h} can again be used to express $\ttu=\ttu(\ttf,\ttg,\tth)$ and $\ttp=\ttp(\ttf,\ttg,\tth)$ 
as functions of the data. This yields the corresponding reduced problem 
\begin{align} \label{eq:redh}
\min_{\ttf,\ttg,\tth} \|\ttu(\ttf,\ttg,\tth) - \ttu^\delta\|_\ttM^2 + \alpha \big( \|\ttf-\ttf^*\|_\ttM^2 + \|\ttg-\ttg^*\|_\ttG^2 + \|\tth-\tth^*\|_\ttH^2\big)
\end{align}
which is quadratic and strictly convex and can be solved by a preconditioned conjugate gradient method.
The computation of the gradients can be realized efficiently via adjoint problems which have a similar structure as \eqref{eq:lin1h}--\eqref{eq:lin2h}; see e.g. \cite{egger11} for more details on the implementation of a related problem.

\subsection{Notes on other filtering approaches} \label{sec:filt}

For later reference, let us also sketch the implementation of the smoothing and the solenoidal filtering approaches 
outlined in the introduction. 
Using the same notation as above, the discrete version of the smoothing filter \eqref{eq:A} 
can be expressed as 
\begin{align} \label{eq:Ah}
\min_\ttu \|\ttu-\ttu^\delta\|_\ttM^2 + \alpha \|\ttu\|_\ttK^2
\end{align}
and the minimizer is characterized by the regularized normal equations 
\begin{align} \label{eq:normalh}
(\ttM + \alpha \ttK) \ttu &= \ttM \ttu^\delta. 
\end{align}
This system is symmetric and positive definite and can be solved efficiently 
by the conjugate gradient method. Parameter robust multigrid preconditioning \cite{schoeberl99}
may be applied to obtain an algorithm of optimal complexity. 

Also the discrete version of the solenoidal filtering approach \eqref{eq:B1}--\eqref{eq:B2} 
can be written in a similar manner. Using the above notation, we obtain
\begin{align} \label{eq:Bh}
\min_\ttu  \|\ttu-\ttu^\delta\|_\ttM^2 + \alpha \|\ttu\|_\ttK^2 \qquad 
\text{s.t. }  \ttB \ttu= 0. 
\end{align}
The optimality system for this constrained minimization problem reads 
\begin{align}
(\ttM + \alpha \ttK) \ttu   + \ttB^\top \ttp &= \ttM \ttu^\delta \label{eq:B1h}\\
\ttB \ttu &= 0. \label{eq:B2h}
\end{align}
The structure of this system is similar to that of the state system \eqref{eq:lin1h}--\eqref{eq:lin2h},
which is why we denote the Lagrange multiplier for the divergence constraint again by $\ttp$ here.
Existence and uniqueness of a solution is guaranteed for all $\alpha \ge 0$, 
if inf-sup stable finite elements are used for the discretization of $\u$ and $\p$. 
Again, iterative methods with multigrid preconditioning may
be applied for the efficient solution \cite{schoeberl98,zulehner02}. 
%

\subsection*{Summary} 
As can be seen from the discussion above, all linear filtering approaches considered in this paper can be discretized systematically and in a uniform framework by finite element methods. The resulting linear optimality systems 
can then always be solved efficiently by iterative solvers and appropriate preconditioning techniques. 
All resulting filters can therefore be considered to be algorithms of optimal complexity.

\section{Computational results} \label{sec:computation} \setcounter{equation}{0}

In order to illustrate the properties of our reconstruction approach, 
we now present some preliminary computational results. 
For ease of presentation, we only consider a simple two-dimensional test problem here.
In all simulations, we use a regular triangulation of the computational domain and 
we assume that the measured velocity field $\ud$ is given at each vertex of the mesh.
For the discretization of the flow equations, we use here the {\sc Mini}  element \cite{boffi08,boffi13}. Other inf-sup stable finite elements, in particular, such leading to exactly divergence free discrete velocity fields, 
could however be used as well.

\subsection{A test problem for channel flow} \label{sec:case1}

We consider the steady laminar flow between two parallel plates already discussed in Section~\ref{sec:poisseuille}.
As a computational domain, we choose here $\Omega=(0,L) \times (0,H)$ with boundaries $\dOin=\{0\} \times (0,H)$, $\dOout=\{L\} \times (0,H)$, and $\dOwall = (0,L) \times \{0,H\}$. 
For our simulations, we set $H=1$ and $L=5$. 
The other model parameters are set to $\nu=0.01$, $p_0=1$, and $p_L=0$. 
The exact solution of \eqref{eq:ex1}--\eqref{eq:ex2} given by the Poisseuille law \eqref{eq:poisseuille}
then reads
\begin{align*} 
\pdag(x,y)=1-x/5 
\qquad \text{and} \qquad 
\udag(x,y) = (10y (1-y),0)
\end{align*}
and the corresponding right hand side and boundary data are 
\begin{align*}
 \fdag(x,y)=(0,0), \qquad \gdag(0,y)=\left(10 y (1-y),0\right), \quad \text{and} \quad \hdag(5,y)=(-50 y^2 (1-y)^2,0).
\end{align*}
These functions will serve as the reference solution and data for our computational tests.

\subsection{Linearization error}
In a first test, we would like to illustrate the estimates for the linearization error given in Theorem~\ref{thm:2}.
To do so, we construct perturbed data $\ud$ by adding random noise to $\udag$ such that $\|\ud-\udag\|_{\L^3(\Omega)}=\delta$,
and then compute the solution $(\u,\p)$ of the linearized problem \eqref{eq:lin1}--\eqref{eq:lin5} with data $\f=\zero$, $\g=\g^\dag$, and $\h=\h^\dag$ by the finite element method outlined above. The resulting errors are displayed in Figure~\ref{fig:linerror}. 

\begin{figure}[ht!]
\includegraphics[width=0.45\textwidth]{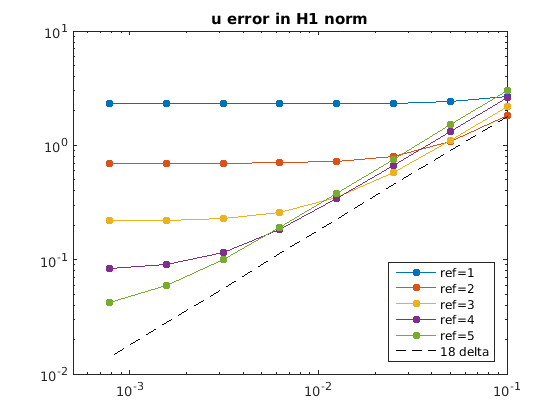}
\hspace*{1em}
\includegraphics[width=0.45\textwidth]{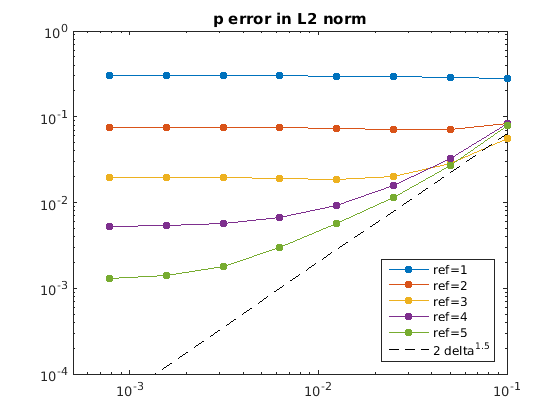}
\caption{\label{fig:linerror}Linearization errors $\|\u-\udag\|_{\H^1(\Omega)}$ (left) and $\|\p-\pdag\|_{L^2(\Omega)}$ (right) for different values of the noise level $\delta$ and various levels of refinement. 
We obtain $O(\delta)$ convergence for the error in the velocity, as stated in Theorem~2 and even $O(\delta^{3/2})$ for the error in the pressure. For large meshsize $h$, a saturation due to discretization errors is observed.}
\end{figure}

The estimate of Theorem~\ref{thm:2} reads $\|\u-\udag\|_{H^1} + \|\p-\pdag\|_{L^2} \le C( \delta + \text{data errors})$. Note that on coarse meshes, the discretization error also contributes to the data error 
and we therefore observe a certain saturation phenomenon as $\delta$ goes to zero. 
A similar behaviour would be obtained in the presence of model errors other than the discretization error.
On refined meshes, the numerical results reveal the expected $O(\delta)$ convergence of the total linearization error as predicted by our theory.

\subsection{Verification of the estimates for the reconstruction error}

Let us next illustrate the two estimates of Theorem~\ref{thm:4}. 
To do so, we compute approximations for the minimizers of \eqref{eq:min1}--\eqref{eq:min2} by solving the discretized optimal control problem \eqref{eq:min1h}--\eqref{eq:min2h}. 
The tests are repeated for different values of the noise level $\delta$. 
Following the remarks in Section~\ref{sec:poisseuillerec}, we should obtain 
\begin{align*}
\|\ua-\udag\|_{\L^2(\Omega)} \le C \delta \qquad \text{and} \qquad \|\ua-\udag\|_{\H^1(\Omega)} \le C(\delta + \delta/\sqrt{\alpha})
\end{align*}
when setting $\f^* = \fdag$, $\g^*=\gdag$, and $\h^*=\hdag$. 
The first estimate does not depend on $\alpha$ while the second predicts a blow-up with $\alpha \to 0$, which is a manifestation of the ill-posedness of the underlying data smoothing problem. 
In the absence of model errors, we should however obtain errors in the size $O(\delta)$ in both, the $L^2$- and the $H^1$-norm, if $\alpha$ is chosen in the order of one.
To avoid the influence of discretization errors, we choose a rather fine mesh for all computations. 
In order to evaluate the influence of errors in the model data, we 
repeat the test with parameters $\f^* \ne \fdag$ such that $\|\f^*-\fdag\|_{\L^2(\Omega)} = 5$. 
%
The results of our numerical test are summarized in Figure~\ref{fig:errest}. 

\begin{figure}[ht!] 
%
\includegraphics[width=0.45\textwidth]{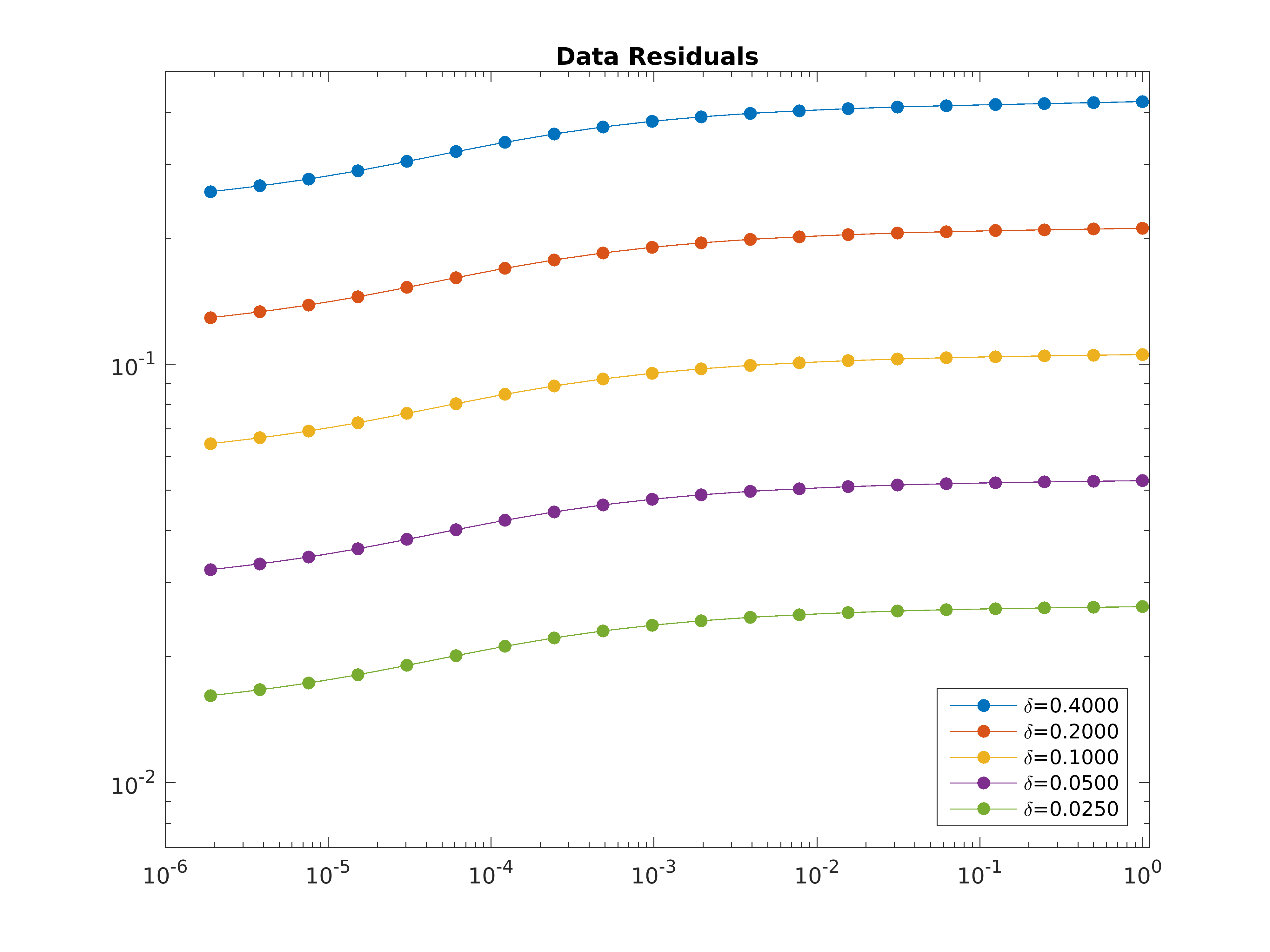}
\hspace*{1em}
\includegraphics[width=0.45\textwidth]{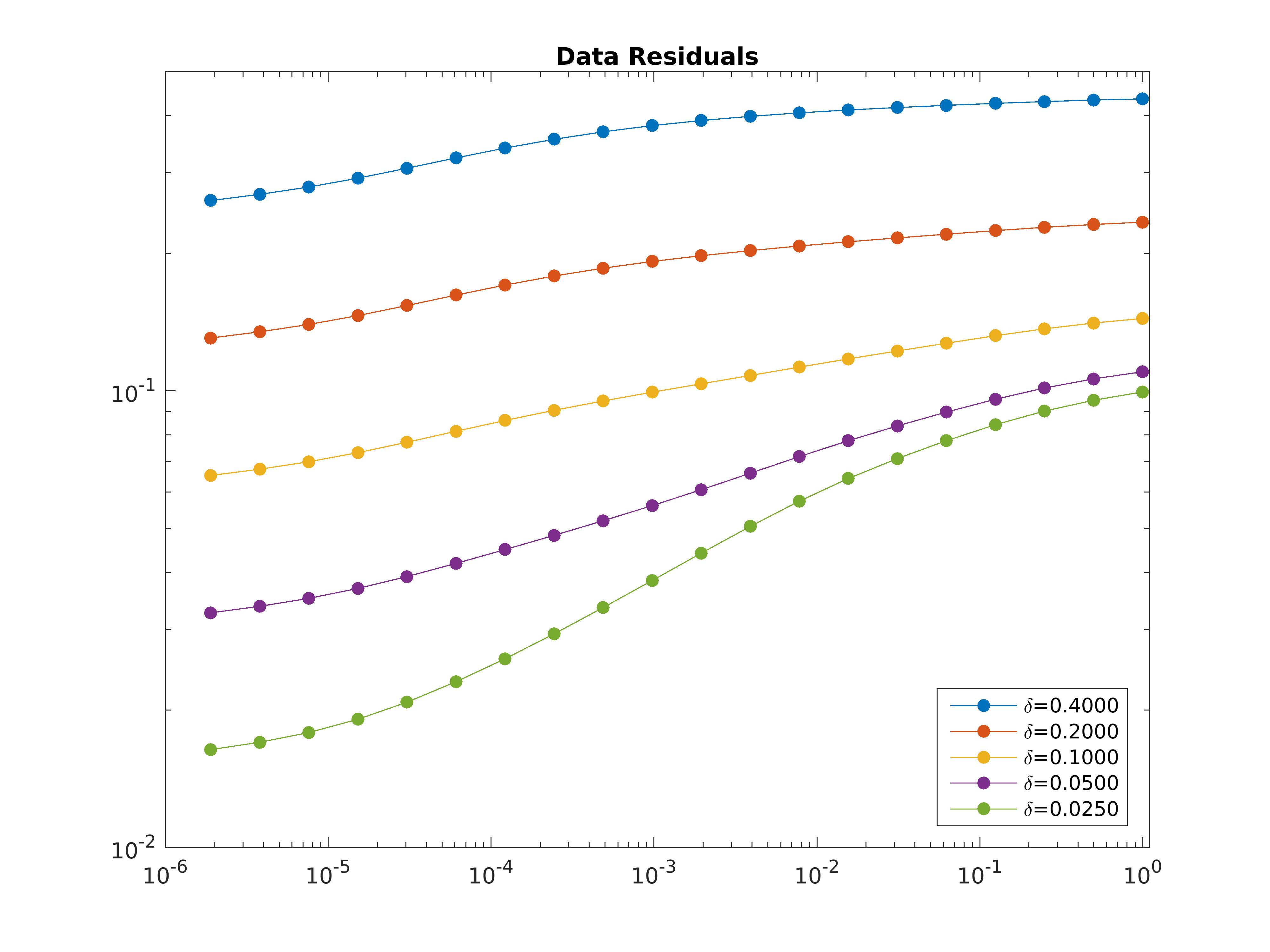}

\includegraphics[width=0.45\textwidth]{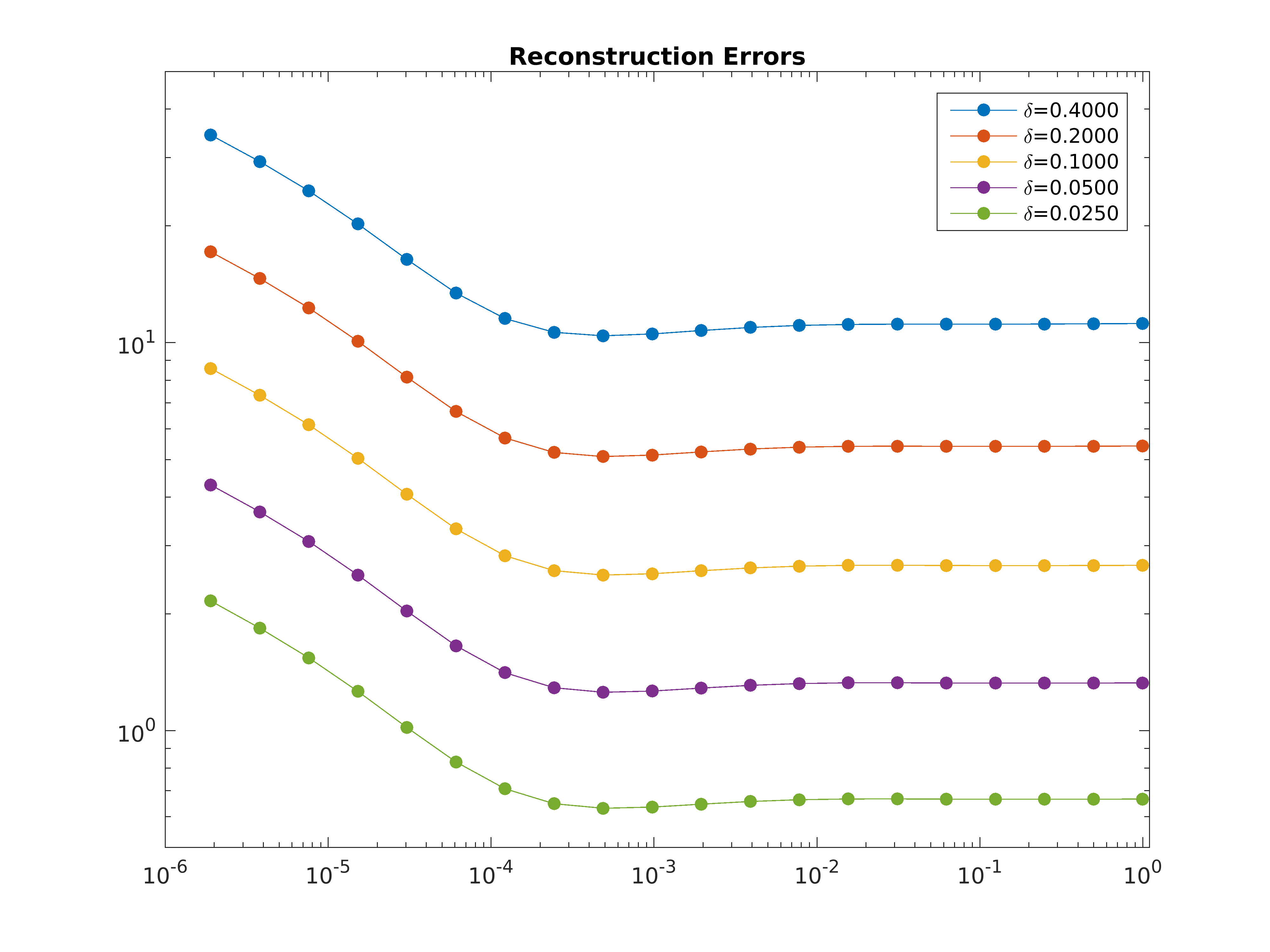} 
\hspace*{1em}
\includegraphics[width=0.45\textwidth]{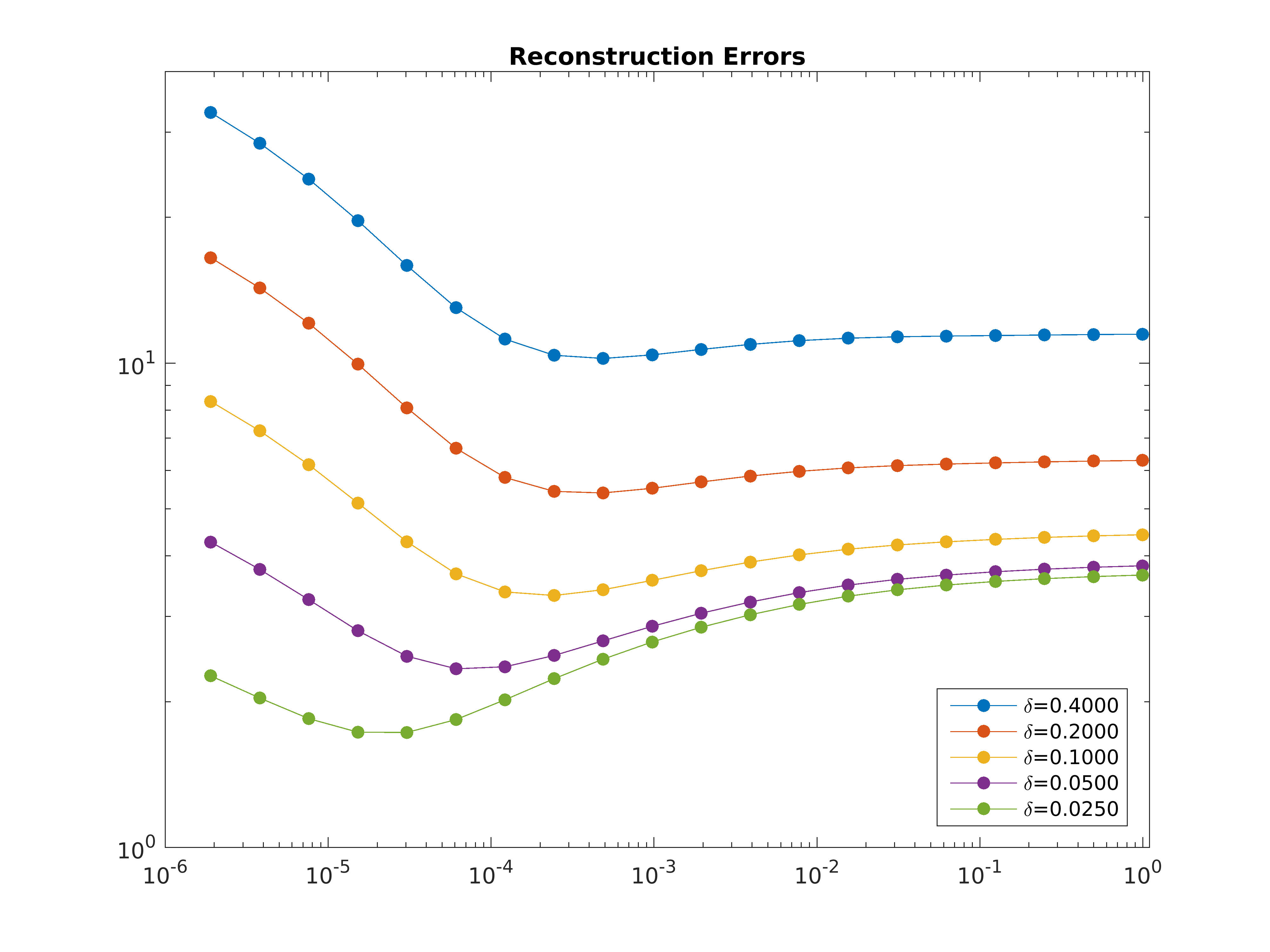}
\caption{\label{fig:errest}Convergence of residuals $\|\ua-\ud\|_{\L^2(\Omega)}$ and reconstruction errors $\|\ua-\udag\|_{\H^1(\Omega)}+\|\pa-\pdag\|_{L^2(\Omega)}$ with $\alpha \to 0$ using 
exact model data $\|\fdag-\f^*\|_{\L^2(\Omega)}=0$ (left) and perturbed model data $\|\fdag-\f^*\|_{\L^2(\Omega)}=5$ (right). 
The individual curves correspond to different values of the noise level $\delta$.}
\end{figure}

As predicted by our analysis, the data residuals $\|\ua-\ud\|_{\L^2(\Omega)}$ decrease monotonically in $\alpha$ and $\delta$, while the reconstruction errors $\|\ua-\udag\|_{\H^1(\Omega)} + \|\pa-\pdag\|_{L^2(\Omega)}$ show the expected semi-convergence behaviour with $\alpha \to 0$. 
These observations are in good agreement with the results of Theorem~\ref{thm:4} and the first assertion of Theorem~\ref{thm:5}.

\subsection{The discrepancy principle}

The monotonicity of the Tikhonov functional stated in Theorem~\ref{thm:5} allows to show that also the data residuals $\|\ua-\ud\|^2_{\L^2(\Omega)}$ are monotonically decreasing with $\alpha \to 0$, which can also be seen 
in the two plots in the first line of Figure~\ref{fig:errest}. 
This motivates the use of a \emph{discrepancy principle} for the choice of the regularization parameter 
in order to achieve an automatic balance of the error contributions due to data noise and model errors. 
Using the standard procedure \cite{engl96}, we define 
$$
\alpha_{dis}(\delta)=\max\{\alpha=\alpha_0 2^{-k} : \|\ua-\ud\|_{\L^2(\Omega)} \le \tau \delta, \quad k \in \NN \},
$$
where $\alpha_0>0$ and $\tau>1$ are some given parameters. 
For our computations, we choose $\alpha_0=1$ and $\tau=1.01$. 
In Table~\ref{tab:alpha}, we list the parameters selected by the discrepancy principle and the corresponding data residuals and reconstruction errors that are obtained for the test case corresponding to the right plot in Figure~\ref{fig:errest}. Note that the model data $\f^* \ne \fdag$ were improperly specified in this example.
For comparison, we also display the results for the values $\alpha_{opt}$ for which the reconstruction errors are minimized.

\begin{table}[ht!]  
\begin{tabular}{c||c|c|c||c|c|c}
$\delta$ &  $\alpha_{dis}$ & residual & error & $\alpha_{opt}$ & residual & error \\
\hline \hline

$0.4$ & $1$ & $0.435366$ & $11.4666$ & $4.8 \cdot 10^{-4}$ & $0.368707$ & $10.2167$ \\ \hline
$0.2$ & $3.1 \cdot 10^{-2}$ & $0.215954$ & $6.13863$ & $4.8 \cdot 10^{-4}$ & $0.185384$ & $5.39017$ \\ \hline
$0.1$ & $3.9 \cdot 10^{-3}$ & $0.10792$ & $3.88025$ & $2.4 \cdot 10^{-4}$ & $0.0905397$ & $3.31047$ \\ \hline
$0.05$ & $4.8 \cdot 10^{-4}$ & $0.0518771$ & $2.66815$ & $6.1 \cdot 10^{-5}$ & $0.0418346$ & $2.33633$ \\ \hline
$0.025$ & $1.2 \cdot 10^{-4}$ & $0.0258675$ & $2.01647$ & $3.0 \cdot 10^{-5}$ & $0.0207808$ & $1.72614$ 

\end{tabular} 
\medskip
\caption{\label{tab:alpha}Data residuals $\|\ua-\ud\|_{\L^2(\Omega)}$ and total reconstruction errors $\|\ua-\udag\|_{\H^1(\Omega)} + \|\pa-\pdag\|_{L^2(\Omega)}$ for  $\alpha_{dis}$ chosen by the discrepancy principle and 
the optimal choice $\alpha_{opt}$ minimizing the reconstruction error.} 
\end{table}

As expected, we observe that the discrepancy principle always chooses the regularization parameter $\alpha_{dis}$ 
somewhat larger than the optimal value $\alpha_{opt}$. 
Repeating the tests with exact model data $\f^* =\fdag$, $\g^* =\gdag$, and $\h^* =\hdag$, 
we obtain $\alpha_{dis} \approx$ constant independent of the noise level $\delta$. 
A similar behaviour is observed for $\alpha_{opt}$. This is in agreement with the first estimate in Theorem~\ref{thm:4} and can be seen in the plots on the left side of Figure~\ref{fig:errest}.

\subsection{Comparison with other filters}

To evaluate the overall performance of our filter, we would like to make also a short comparison 
with the other filtering approaches discussed in the introduction. 
As outlined in Section~\ref{sec:fem}, these can be implemented in a similar manner 
as the method presented in this paper which allows a fair comparison. 

In the following tests, we compare the smoothing filter $\eqref{eq:A}$, the solenoidal filtering  \eqref{eq:B1}--\eqref{eq:B2} with and without smoothing, and the fluid-dynamically consistent filter \eqref{eq:min1}--\eqref{eq:min2} presented in this paper. 
%
Whenever required, the regularization parameter $\alpha$ is selected via the discrepancy principle with $\tau=2$.
In Table~\ref{tab:comparison1}, we list various measures for the reconstruction error for the different choices of the filter. 
%

\begin{table}[ht!] 
\begin{tabular}{l||r|r|r|r}
method & $\|\u-\udag\|_{L^2}$ & $\|\u-\udag\|_{H^1}$ & $\|\p-\pdag\|_{L^2}$ & $\|\text{div}_h \u\|_{L^2}$ \\
\hline \hline

smoothing               & 0.119116 &  3.983561  & *.**** & 0.625567 \\ \hline
solenoidal ($\alpha=0$) & 0.081791 & 20.662327  & 1.290827 & 0.000000 \\ \hline
solenoidal w. smoothing & 0.117853 &  3.921923  & 1.290828 & 0.000000 \\ \hline
fluid-dyn. consistent   & 0.058466 &  3.066198  & 0.073987 & 0.000000

\end{tabular} 
\medskip
\caption{\label{tab:comparison1}Reconstruction errors and discrete divergence for different filters at noise level $\delta=0.1$ on 
a mesh with $9153$ 
vertices and $17920$ 
elements.}
\end{table}

The discrete divergence $\text{div}_h \u$ here corresponds to the projection of $\text{div} \, \u$ onto the discrete pressure space. 
All filters yield small errors for the velocity in the $L^2$-norm which is an immediate consequence of their construction. 
The $H^1$-norm of the error in the velocities is comparable for the filters involving some sort of smoothing. For the solenoidal filter without smoothing, the $H^1$-norm errors increase with decreasing meshsize
due to the ill-posedness of the reconstruction problem.  
The smoothing filter yields a velocity reconstruction which is not discrete divergence free and no information about the pressure is obtained. 
The three filters involving a divergence constraint yield some reconstruction of the pressure. Those obtained with the fluid-dynamically consistent filter are however by an order of magnitude better than those obtained with the solenoidal filters. 
In summary, the fluid-dynamically consistent filter proposed in this paper yields the best reconstruction of the flow fields with respect to all error measures listed in the table.

\section{Discussion} \label{sec:discussion} \setcounter{equation}{0}

In this paper we considered the reconstruction of the velocity and pressure fields of an incompressible fluid from distributed measurements of the flow velocities. 
For the stable solution of this inverse problem, we considered a novel filter which minimizes a 
weighted sum of the data residual and the mismatch of a specified flow model. This strategy was
formulated as an optimal control problem constrained by the prescribed flow model.

In order to guarantee the well-posedness of our approach, we utilized a linearized flow model which directly incorporated the measured velocity field. This allowed us to show the existence and uniqueness of minimizers and to derive estimates for the reconstruction errors in various norms. The theoretical results were illustrated by numerical tests including a comparison to other filters discussed in the literature. 

The strategy of using a linearized flow model as a constraint in the reconstruction process could be generalized in various ways: 
While we directly used the velocity measurements in order to specify our 
linearized flow model here, some pre-filtered velocity field could be used as well. 
Our analysis also covers this case and could possibly be refined leading to sharper estimates. 
Repeating the argument, one could also define an incremental reconstruction approach. 
Preliminary numerical tests for such multi-step algorithms showed a further significant improvement of the reconstructed flow fields. A full analysis would however exceed the scope of the current presentation. 
In order to handle more general flow regimes, some sort of turbulence model should be incorporated as a next step 
and a refined modeling of the constitutive equations and the boundary conditions should be considered.  
Both aspects are subject of current research by the authors.

\section*{Acknowledgements}
The authors would like to gratefully acknowledge the support by the German Research Foundation (DFG) via grants IRTG~1529, GSC~233, and TRR~154. Part of the work of the second author was carried out during a research stay at Waseda University, Tokyo. The hospitality and kind support of Waseda university is gratefully acknowledged.


\small

\end{document}